\documentclass{amsart}
\usepackage[utf8]{inputenc}
\usepackage{amssymb,mathscinet,stmaryrd,phonetic}
\usepackage{enumitem,mathrsfs,hyperref,cmll,color}
\usepackage{latexsym,amsthm,amsmath}
\newif\ifdraft%\drafttrue
\ifdraft
    \usepackage{showlabels}
\fi

\def\seq#1_#2{\langle #1_#2:#2\in\omega\rangle} % sequence notation
\def\set#1:#2.{{\{\,#1: #2\,\}}} % set notation
\newcommand\la[1]{{\langle{#1}\rangle}}
\newcommand\ain{\subseteq^*} % almost inclusion
\newcommand\cl[1]{\overline{#1}} % closure notation
\newcommand\N{{\mathbb N}}
\newcommand\Q{{\mathbb Q}}

\renewcommand\L{{\mathbb L}}

\newcommand\G{{\mathbb G}}
\newcommand\K{{\mathcal K}}

\newcommand\I{{\mathcal I}}
\newcommand\D{{\mathcal D}}
\newcommand\rank{\mathop{\rm rank}}
\newcommand\Int{\mathop{\rm Int}}
\newcommand\so{\mathop{\mathfrak{so}}}
\newcommand\nwd{\mathop{\bf nwd}}
\newcommand\csct{\mathop{\bf csc}}
\newcommand\cpt{\mathop{\bf cpt}}
\newcommand\scl{\mathop{\rm scl}} % scatteredness level of a point in a set
\newcommand\kw{\ensuremath{k_\omega}}
\newcommand\cw{\ensuremath{c_\omega}}

\newcommand\uG{{1_{\G}}}
\newcommand\zG{{0_{\G}}}
\newcommand\C{\ensuremath{\mathfrak c}}

\newcommand\restr[2]{{#1}|_{#2}}
\newcommand\vD{vD}
\newcommand\IIA{{\sf IIA}}
\newcommand\UIA{{\sf UIA}}

\newtheorem{theorem}{Theorem}
\newtheorem{lemma}[theorem]{Lemma}
\newtheorem{example}[theorem]{Example}
\newtheorem{corollary}[theorem]{Corollary}

\newtheorem{definition}[theorem]{Definition}
\newtheorem{question}{Question}
\newtheorem{prop}[theorem]{Proposition}

\newlist{couup}{enumerate}{10}
\setlist[couup]{label={\rm(\arabic*)}, ref=(\arabic*)}

\title{Invariant Ideal Axiom}

\author{Michael Hru\v{s}\'ak}
\address{Centro de Ciencias Matem\'aticas, Universidad Nacional Aut\'onoma de M\'exico, Campus Morelia, Morelia, Michoac\'an, 58089, M\'exico}
\curraddr{}
\email{michael@matmor.unam.mx}
\thanks{The research of the first author was supported  by PAPIIT
  grants IN100317 and IN104220, and CONACyT grant A1-S-16164.}
\urladdr{http://www.matmor.unam.mx/~michael}

\author{Alexander Shibakov}
\address{Department of Mathematics, Tennessee Tech. University, 110 University Drive, Cookeville, Tennessee 38505, USA}
\curraddr{}
\email{ashibakov@tntech.edu}
\thanks{}
\urladdr{http://www.math.tntech.edu/alex}


\date{March 2021}

\begin{document}
\begin{abstract}
We introduce and prove the consistency of a new set theoretic axiom
we call the \emph{Invariant Ideal Axiom}.  The axiom enables us to
provide (consistently) a full topological classification of countable sequential
groups, as well as fully characterize the behavior of their finite
products. 

We also construct examples that demonstrate the optimality of the
conditions in \IIA, and list a number of open questions.

\end{abstract}
\maketitle
\section*{Introduction}
In this paper we introduce a new set theoretic principle we call the
\emph{Invariant Ideal Axiom} (\IIA\ for short) and prove its
consistency with the usual axioms of ZFC.  As the main
application of \IIA\ we show that it implies that all countable
sequential groups are either metrizable or $\kw$ and, in particular,
every countable sequential group has a definable (in fact
$F_{\sigma\delta}$) topology,  thus concluding the project initiated
in the 1970's of determining the structural theory of countable
Fr\'echet and sequential groups (see~\cite{arkhangelskii-questions,
  arkhangelskii-malykhin, comfort-problems, gerlits-nagy, 
  malykhin, malykhin-shakhmatov, nyikos, nyikos-cantor, nyikos92,
  zelenyuk} for some early papers on the subject). This line of
research is considered central in topological algebra
(see~\cite{banakh-zdomskyy, chasco-angelic, hrusak-ramos-precompact,
  shakhmatov-shibakov, shibakov-96, shibakov-96a, shibakov-99,
  todorcevic-product, todorcevic-uzcategui, todorcevic-uzcategui-k},
and several others) with some of the early questions answered only
recently (see~\cite{brendle-hrusak, hrusak-ramos-malykhin, shibakov-CH, shibakov-nointeresting}).
For a more comprehensive overview of the field, including the history,
the fundamental results, and the open problems the reader may wish
to consult several excellent surveys available on the subject, as well
as a number of books on topological
algebra:~\cite{arkhangelskii-tkachenko, dikranjan-shakhmatov, kakol-descriptive,
  moore-todorcevic, shakhmatov-survey}.

The axiom is fully accessible to the mathematicians working in topology or algebra and does
not require any knowledge of modern set theory. Aside from giving the
ultimate structural result for countable sequential groups, the axiom
has a profound impact on product properties of sequential groups. 

Our hope and expectation is that the axiom \IIA\ provides both a
canonical environment, and a test model for  future study of
convergence in topological algebra. 

To better relate our results to the existing body of research one may
recall that arguments in analysis and topology often depend on establishing the extent
of various special classes of topological spaces. Fusing algebraic and topological properties
proved to be among the most fruitful techniques. Classical examples of such results are
the implication $T_1\Rightarrow T_{3{1\over2}}$ in general topological groups
(or even $T_3\Rightarrow T_{3{1\over2}}$ in \emph{para}topological groups,
see~\cite{banakh-ravsky}) and the
\emph{Birkhoff-Kakutani theorem} (see~\cite{arkhangelskii-tkachenko}
for these and other facts about topological groups) on the metrizability of
first countable $T_1$ topological groups. Metrizability theorems 
in particular drew a lot of attention, stimulating a search for the
weakest set of conditions that jointly imply that a given topology
is generated by a metric.

In the class of topological groups, compactness and countable
tightness together imply metrizability
(see~\cite{arkhangelskii-tkachenko}) so it is natural to look
for a similar yet less restrictive set of conditions that may yield the
same result. Reasoning along these lines led V.~Malykhin to ask
about the existence of countable (equivalently, separable) Fr\'echet
non metrizable groups (see~\cite{arkhangelskii-questions}).

Malykhin's problem generated a large body of research that illustrates
another important quality shared by several results in this
area. Namely, the effect of a particular set of restrictions is
greatly influenced by set theory. As a case in point, Malykhin's
problem has an affirmative answer in a variety of set theoretic
universes, including models of MA. The conclusive result, establishing
the independence of the answer to Malykhin's problem of the axioms of
ZFC was obtained by the first author
and~U.~A.~Ramos-Garc\'{\i}a in~\cite{hrusak-ramos-malykhin}
using a forcing construction. The same paper also contains a
construction of a countable Fr\'echet non metrizable group under
a very weak set theoretic assumption $\diamondsuit(2=)$. 

Malykhin asked (see~\cite{zelenyuk-protasov-abelian}) a related
question about the class of countable sequential abelian groups
(see below for all the appropriate definitions).
This question was fully solved in~\cite{zelenyuk-protasov-abelian} by E.~Zelenyuk and
V.~Protasov who established (in ZFC) the existence of countable sequential
group topologies that are not Fr\'echet on any infinite countable
abelian group. The existence of such topology on \emph{nonabelian}
countable groups (specifically, the free group) was well-known
(see~\cite{ordman-thomas}). 

The investigation into the class of sequential groups prompted
P.~Niykos (see~\cite{nyikos}) to look at their \emph{sequential order},
which can be roughly thought of as the ordinal measure of the
complexity of the closure operator in sequential spaces. The existence
of sequential groups (of any size) of sequential order strictly
between 1 and $\omega_1$ turned out to be independent of the axioms of
ZFC, as well (see~\cite{shibakov-96a}, \cite{shibakov-nointeresting},
and~\cite{shibakov-cohen}).

A thorough review of existing ZFC constructions of sequential non Fr\'echet
groups (see~\cite{kakol-descriptive}, \cite{tkachenko87},
\cite{zelenyuk-protasov-abelian}, \cite{chasco-angelic}) reveals a
structure common to all such examples. Their topology 
is determined by a countable family of (countably)
compact subspaces (i.e.~is $\kw$, see~\cite{franklin-thomas_kw} and the
definition below). Perhaps the most widely
known family of $\kw$ spaces is the class of countable CW-complexes
(see~\cite{bredon}). Various results in algebraic topology (such as the homotopy
equivalence for filtered spaces theorem of J.~Milnor, see~\cite{boardman-vogt})
heavily depend on the $\kw$ property. The class of $k_\omega$ spaces
is well behaved, in particular, it is finitely productive, and every
countable $k_\omega$ space is sequential and analytic (see below for
the definitions and further discussion).

In~\cite{shibakov-analytic}, answering a question of S.~Todor\v
cevi\'c and C.~Uzc\'ategui, the second 
author showed that at least in the definable case (more specifically,
in the class of countable \emph{analytic} groups), the only
sequential examples of countable groups are $\kw$ or metrizable. This
naturally brought about the question (posed
in~\cite{shibakov-nointeresting}) whether it is consistent with ZFC
that all countable sequential groups are either metrizable or $\kw$
(equivalently, whether all countable sequential groups are analytic).

The main tool introduced in this paper, \emph{the Invariant Ideal
  Axiom}, or \IIA, is used to answer this question in the
affirmative. As important corollaries, we show that \IIA\ generates a
complete classification of sequential group topologies on countable
groups, as well as allows for a transparent description of products of
such groups.

To support our claim of the optimal nature of \IIA\ in the study of
convergence in countable groups, we show that its natural
generalization fails to stay consistent. We also construct  an example
demonstrating the differences between the case of countable sequential
groups and their separable counterparts. We conclude by listing a few
open questions we believe will lead to greater insight about this
field of research.

\section{Preliminaries}
All topological spaces and groups considered are $T_1$ and completely
regular. To see more about topological groups consult
\cite{arkhangelskii-questions, arkhangelskii-tkachenko,
  comfort-handbook, comfort-problems, dikranjan-shakhmatov,
  shakhmatov-survey, tkachenko-01}.

Recall that a topological space $X$ is \emph{Fr\'echet} if for any
$x\in\overline{A}\subseteq X$ there is a sequence $S\subseteq A$ such
that $S\to x$.  A space $X$ is \emph{sequential} if for every
$A\subseteq X$ which is not closed there is a $C\subseteq A$ such that
$C\to x\not\in A$.  The term `Fr\'echet space' appears to have been
coined by Arkhangel'skii in \cite{arkhangelskii-frechet}, while the
term `sequential' appears for the first time in Franklin's
\cite{franklin}, where the following notion is defined:

Given $A\subseteq X$, define the \emph{sequential closure} of $A$ as
$$[A]' = \set x\in X : C\to x\text{ for some }C\subseteq A.\text{, and then  recursively}$$
$$[A]_0=A\text{ and }[A]_\alpha = \cup\set[[A]_\beta]' : \beta<\alpha.\text{ for }\alpha\leq\omega_1.$$
Then $X$ is sequential if and only if $\overline{A}=[A]_{\omega_1}$
for every $A\subseteq X$, and the \emph{sequential order} of $X$ is
defined as
$$\so(X) = \min\set\alpha \leq \omega_1 :
  [A]_\alpha=\overline{A}\text{ for every }A\subseteq X..$$
Fr\'echet spaces are easily seen to be exactly those sequential spaces $X$
for which $\so(X)\leq1$. The following definition plays a central
role in our investigation.

\begin{definition}\label{kw.def}
A topological space $X$ is called a \emph{\kw-space\/} (\emph{\cw-space})
if there exists a countable family $\K$ of (countably) compact
subspaces of $X$ such that a $U\subseteq X$ is open in $X$ if and only
if $U\cap K$ is relatively open in $K$ for every $K\in\K$.
\end{definition}

Countable $\kw$ spaces are always sequential and the class of
\kw-spaces is productive. Such spaces are definable objects and have
$F_{\sigma\delta}$ topologies. 

What follows is a short discussion of \emph{test spaces}:
\begin{itemize}
\item {\bf Arens space} (\cite{arens}): $S_2=[\omega]^{\leq 2}$ where
  $U\subseteq S_2$  is open if and only if for every $s\in U$ such that $|s|<2$ the set
  ${\set s\cup\{ n\}\in S_2 : s\cup\{n\}\not\in U.}$ is finite,
\item {\bf sequential fan} (\cite{alexandroff-hopf}): the quotient
  $S(\omega)=S_2/[\omega]^{\leq 1}$, and finally
\item {\bf convergent sequence of discrete sets} (\cite{vandouwen}):
  $D(\omega)=\omega\times\omega\cup\{(\omega,\omega)\}\subseteq(\omega+1)^2$
  in the natural product topology.
\end{itemize}

The sequential fan $S(\omega)$ and the convergent sequence of discrete
sets $D(\omega)$ are both Fr\'echet spaces,  $D(\omega)$ is
metrizable while $S(\omega)$ has character $\mathfrak d$.
The space $S_2$ is sequential, and $\so(S_2)=2$. Both $S(\omega)$ and
$S_2$ are \kw -spaces while $D(\omega)$ is not.

\begin{prop}\label{ts.properties}\ 
\begin{enumerate}
%\item {\rm (P.~Nyikos \cite{nyikos})} A topological group contains a
%  copy of $S_2$ if and only if it contains a copy of $S(\omega)$.
\item {\rm (Y.~Tanaka \cite{tanaka})} A sequential space contains a
  copy of $S(\omega)$ if and only if it contains a closed copy of
  $S(\omega)$. 

\item {\rm (Y.~Tanaka \cite{tanaka})} A countable sequential
  topological group is Fr\'echet if and only if it does not contain a
  closed copy of $S(\omega)$.

\item {\rm (T.~Banakh and L.~Zdomsky\u{\i} \cite{banakh-zdomskyy})} 
  If a topological group $\G$ contains closed copies of $S(\omega)$
  and $D(\omega)$, it also contains a subset $D$ such that $D$ is not
  closed in $\G$ and is almost disjoint from every convergent sequence
  in $\G$.

\end{enumerate}
\end{prop}

Recall  that the \emph{Cantor-Bendixson} derivative $A'$  of a
topological space $A$ is defined by $A'=A\setminus \set x: x\text{ is an
isolated point of $A$}.$. The Cantor-Bendixson derivative can be
iterated---recursively define $(A)_\alpha$ for any ordinal $\alpha$ by
putting $(A)_0= A$, $(A)_{\alpha+1}=(A)_{\alpha}'$ and
$(A)^\lambda=\bigcap_{\alpha<\lambda} (A)_{\alpha}$ for $\lambda$
limit. 

The \emph{full Cantor-Bendixson derivative} (also called the
\emph{Cantor-Bendixson} (or \emph{perfect}) \emph{kernel}) of a space $A$ is $(A)_\alpha$ where
$\alpha$ is an ordinal such that $(A)_\alpha=(A)_\beta$ for any
$\beta\geq\alpha$.

A topological space $A$ is \emph{scattered} if every subset of $A$
contains an isolated (in the subset) point, equivalently
if its full Cantor-Bendixson derivative is empty.
Every scattered space is thus naturally stratified into levels,  $x\in
A$ belonging to the \emph{$\alpha$-th level} (denoted by
$\scl(x,A)=\alpha$) if and only if  $\alpha$ is the unique ordinal
such that $x\in(A)_{\alpha}\setminus(A)_{\alpha+1}$. The \emph{height} of $A$
($\scl(A)$) is the smallest ordinal $\alpha$ such that
$(A)_\alpha=\varnothing$. 

Throughout the paper, $\csct(X)$ denotes the ideal generated by closed scattered
subsets of $X$, $\nwd(X)$ is the ideal of nowhere dense
subsets of $X$,  and $\cpt(X)$ stands for the ideal generated by all
the compact subsets of $X$.

\section{The Invariant ideal Axiom}\label{IIA}

\subsection{Introducing \IIA} Analyzing the proofs of 

\begin{theorem}[\cite{hrusak-ramos-malykhin}]\label{malykhin}
It is consistent that every countable Fr\'echet group is metrizable.
\end{theorem}

\begin{theorem}[\cite{shibakov-nointeresting}]\label{no-inter}
It is consistent that every separable sequential group is either metrizable or
has sequential order $\omega_1$.
\end{theorem}

\noindent we have isolated the Invariant Ideal Axiom \IIA\ which we shall
present next.

\smallskip

First let us introduce the relevant notation. Recall that an \emph{ideal} is a
family $\mathcal I\subseteq \mathcal P(\G)$ closed under taking
subsets and finite unions, and it is \emph{invariant} if both $g\cdot
I=\{g\cdot h: h\in I\}$ and $I\cdot g=\{h\cdot g: h\in I\}$, as well as
$I^{-1}=\{h^{-1}: h\in I\}$ are in $\mathcal I$ for every
$I\in\mathcal I$ and $g\in \G$. We shall assume throughout the paper that all ideals contain all finite subsets of $\G$.
Recall also that $\mathcal
I^+=\mathcal P(\G)\setminus \mathcal I$. Given a point $x$ in a
topological space (or a topological group) we denote by
$$\mathcal I_x=\{A\subseteq X: x\not\in \cl{A}\}$$ 
the dual ideal to the filter of neighborhoods of $x$.
An ideal $\I$ on a set
$X$ is \emph{$\omega$-hitting} if for every countable family $\mathcal
Y$ of infinite subsets of $X$ there is an $I\in\mathcal I$ such that
$Y\cap I$ is infinite for every $Y\in\mathcal Y$.

We call an ideal
$\mathcal I$ \emph{tame} if for every $Y\in\mathcal I^+$ every
$f:Y\to\omega$ and every $\omega$-hitting ideal $\mathcal J$ on
$\omega$ there is a $J\in\mathcal J$ such that $f^{-1}[J]\not \in
\mathcal I$, i.e. if no ideal \emph{Kat\v{e}tov-below} a restriction
of $\mathcal I$ to a positive set is $\omega$-hitting (see
e.g. \cite{hrusak-filters} for more on Kat\v{e}tov order and
$\omega$-hitting ideals). Finally, we call an ideal $\mathcal I$ of
subsets of a topological group $\G$ \emph{weakly closed} if for every
set $A\subseteq \G$ and every sequence $C\subseteq \G$ convergent to $1_\G$,
$$A\in\mathcal I\text{ if and only if } A\cup \{x: C\cdot x \subseteq^* A\}\in\mathcal I.$$
It is immediate from the definition that every ideal generated by
(sequentially) closed subsets of $\G$ is weakly closed, in particular the
ideals $\csct(\G)$, $\nwd(\G)$  and $\cpt(\G)$ are all invariant weakly
closed ideals in any topological group.

\smallskip

We call a subset $Y$ of a topological space $X$ \emph{entangled} if
$\mathcal I_x\restriction Y$ is $\omega$-hitting for every $x\in X$. We
shall call a topological space $X$ \emph{groomed} if it does not
contain a dense entangled set.

The class of groomed spaces includes all non-discrete Fr\'echet and sequential
spaces, as well as all non-discrete \emph{subsequential} spaces
(i.e.~subspaces of sequential spaces).

\begin{lemma}
Every non-discrete subsequential space is groomed.
\end{lemma}
\begin{proof}
Let $X$ be a dense subspace of a sequential space $Y$. Let $D\subseteq
X$ be dense and let $x\in X$ be a point which is not isolated. As $Y$
is sequential, there are countably many disjoint infinite subsets
$\{C_n:n\in\omega\}$ of $D$ such that 
\begin{enumerate}
\item Each $C_n$ converges to some point $x_n\in Y$ (not necessarily distinct), and
\item For every neighborhood $U$ of $x$  there are infinitely many
  $n\in\omega$ such that $x_n\in U$.
\end{enumerate}
This is easily proved by induction on the sequential order of $x$ in $D$:

If there is a sequence $C\subseteq D\setminus \{x\}$ convergent to $x$
let $\{C_n: n\in\omega\}$ be any collection of infinite pairwise
disjoint subsets of $C$. 

For the inductive step, assume that $x\in Y$ is the limit of a
convergent sequence $\{x_n:n\in\omega\}$ such that for each $x_n$ (by
the inductive hypothesis) exist pairwise disjoint sequences $\{C^n_m:
m\in\omega\}$ of elements of $D$ convergent each to a point $x^n_m$ such that every open
set $U$ containing $x_n$ contains infinitely many of the $x^n_m$. 
Let $\{D_m^n: n,m\in\omega\}$ be a disjoint refinement of $\{C_m^n:
n,m\in\omega\}$. Then it is a collection of pairwise disjoint
sequences convergent in $Y$ and every neighborhood $U$ of $x$ will
contain all but finitely many of the   
$\{x_n:n\in\omega\}$, and, consequently, infinitely many of the
$\{x_m^n: n,m\in\omega\}$. 

Then, however, $D$ is not entangled, as $x\in\overline Z$ for every $Z$
such that $|Z\cap C_n|=\omega$ for every $n\in\omega$. 
\end{proof}

Examples of spaces which are not groomed are discrete spaces, $\omega^*$ and
$2^\C$. 

\medskip

We are now ready to introduce the {\it Invariant Ideal Axiom}:

\bigskip
\noindent
\IIA: {\it For every countable groomed topological group $\G$ and
  every tame, weakly closed  invariant ideal $\mathcal I\subseteq 2^\G$ one of the
  following holds:
  
\smallskip
\begin{enumerate}
    \item\label{seq.capture} there is a countable $\mathcal S\subseteq
      \mathcal I$ such that for every infinite sequence $C$ convergent
      in $\G$ there is an $I\in \mathcal S$ such that $C\cap I$
      is infinite,
    
    \smallskip
    
    \item\label{almost.pi} there is a countable $\mathcal H\subseteq
      \mathcal I^+$ such that for every non-empty open $U\subseteq
      \G$ there is an $H\in \mathcal H$ such that $H\setminus
      U\in \mathcal I$.
\end{enumerate}}

\smallskip

We refer to the $\mathcal S$ from the first alternative as a \emph{sequence capturing} set, and the set 
$\mathcal H$ from the second alternative as an \emph{almost $\pi$-network}.

\medskip

To see the relevance of the axiom let us deduce the solution to the
Malykhin problem from it. We first recall the following simple lemma
from the literature (we include the short proofs for the sake of
completeness):

\begin{lemma}\label{nwd-lemma} Let $X$ be a countable Fr\'echet space without isolated points.
\begin{enumerate}
    \item {\rm (\cite{barman-dow})} If $x\in X$ and $\mathcal X$ is a
      countable collection of nowhere dense subsets of $X$ then
      there is a $C\subseteq X$ convergent to $x$ such that $X\cap N$
      is finite for every $N\in \mathcal X$.
    \item {\rm (\cite{hrusak-ramos-malykhin})} The ideal
      $\nwd(X)$ of nowhere dense subsets of $X$ is tame.
\end{enumerate}
\end{lemma}
\begin{proof}

Ad (1): Enumerate $\mathcal{X}$ as $\langle M_{n} \colon n \in \omega
\rangle$. As $X$ is Fr\'echet without isolated points, there is a
one-to-one sequence $\langle x_{n} \colon n \in \omega \rangle
\subseteq X \setminus \{x\}$ converging to $x$.  For each $n\in\omega$
the set $X_{n}=X \setminus \left(\{x\} \cup \bigcup_{i<n}M_{i}\right)$
is dense in $X$, so using the Fr\'echet property again, there is for
each $n \in \omega$ a sequence $\langle y_{i}^{n} \colon i \in \omega
\rangle \subseteq X_{n}$ converging to $x_{n}$. Then
$x\in\overline{\{y_{i}^{n} \colon i,n \in \omega\}}$, hence there is a
sequence $C\subseteq \{y_{i}^{n} \colon i,n \in \omega\}$ converging
to $x$.  Now, $C\cap M_n\subseteq \{y_{i}^{m} \colon i\in\omega,
\ m\leq n\}$ and as each sequence $\langle y_{i}^{n} \colon i \in
\omega \rangle \subseteq X_{n}$ converges to $x_{n}\neq x$, $C\cap
M_n$ is finite for every $n\in\omega$.

\smallskip 

To see (2), let $Y \in \mathsf{nwd}^{+}(X)$, let $f: Y\to\omega$, and
let an $\omega$-hitting ideal $\mathcal J$ on $\omega$ be given. Put
$Z=  \text{Int}(\overline{Y})\cap Y$. Then either 
\begin{enumerate}
\item[(a)] there is an $n\in \omega$ such that $f^{-1}(n) \in \mathsf{nwd}^+(X)$, or

\item[(b)] for every $x \in Z$ there is a sequence $C_{x} \subseteq
  Z\setminus \{x\}$ converging to $x$ such that $\restr{f}{C_{x}}$ is
  finite-to-one. 
\end{enumerate}

If there is an $n\in \omega$ such that $f^{-1}(n) \in \mathsf{nwd}^+(X)$, let $J=\{n\}\in\mathcal J$.
If, on the other hand, $f^{-1}(n) \in \mathsf{nwd}(X)$ for all $n \in
\omega$, apply (1) to $Z$ and every $x\in Z$ with $\mathcal X=
\{ f^{-1}(n) \cap Z: n\in\omega\}$ to get $\{C_x: x\in Z\}$ as in
(b). The family$\{f[C_x]: x\in Z\}$ is then a countable collection of
infinite subsets of $\omega$, hence there is a $J\in\mathcal J$ such
that $J\cap f[C_x]$ is infinite for every  $x\in Z$. Then $f^{-1}[J]$
is dense in $Z$, hence in either case $J$ is an element of $\mathcal
J$ such that $f^{-1}[J] \in \mathsf{nwd}^+(X)$. 
\end{proof}

\begin{theorem}\label{just.Frechet}
Assuming \IIA, every separable Fr\'echet group is metrizable.
\end{theorem}

\begin{proof}
Let $\mathbb H$ be a separable Fr\'echet group and let $\G\subseteq
\mathbb H$ be a dense countable subgroup. Apply \IIA~to
$\G$ and $\nwd(\G)$. Alternative (1) fails by
Lemma \ref{nwd-lemma} so there is a countable family $\mathcal X$ of
somewhere dense subsets of $\G$, hence also somewhere dense in
$\mathbb H$ such that every open set contains $\bmod \nwd$ an element
of $\mathcal X$. Then
$$\{\mathrm{int} (\overline X): X\in \mathcal X\},$$
where the interior and closure are taken in $\mathbb H$, form a countable
$\pi$-base, and as $\pi$-weight and weight coincide in topological
groups, the group $\mathbb H$ is second countable hence metrizable.
\end{proof}

\begin{corollary} Assuming \IIA, $\mathfrak p=\omega_1$ and $\mathfrak b>\omega_1$.
\end{corollary}

\begin{proof}
It is well known that if either $\mathfrak p>\omega_1$ or $\mathfrak
b=\omega_1$ then there is a separable non-metrizable Fr\'echet group
see e.g.\ \cite{nyikos, nyikos-cantor, orenshtein-tsaban}.
\end{proof}

The next remark we would like to make is that the assumption that the
group is groomed cannot be dropped from the statement of \IIA, as is
shown in the proposition below.

\begin{prop}
There is a countable topological group $\G$ and a tame, weakly closed,
invariant ideal $\mathcal I\subseteq \mathcal P(\omega)$ such that
\IIA\ fails for $\G$ and $\mathcal I$.
\end{prop}

\begin{proof} Without loss of generality we can assume  \IIA\ as its failure
  provides an example. So, in particular, we can assume that $\mathfrak b>\omega_1$.

\smallskip

Let $\{A_\alpha, B_\alpha:\alpha<\omega_1\}\subseteq[\omega]^\omega$
be a \emph{Hausdorff gap}, i.e.\
\begin{enumerate}
    \item $A_\alpha\subseteq^* A_\beta\subseteq^* B_\beta\subseteq^*
      B_\alpha$ for $\alpha<\beta<\omega_1$, and
    \item there is no $X$ such that $A_\alpha\subseteq^* X\subseteq^*
      B_\alpha$ for every $\alpha<\omega_1$.
\end{enumerate}
Topologize the group $\G=[\omega]^{<\omega}$ by declaring the
sets (in fact, subgroups) $[F]^{<\omega}$ open neighborhoods of
$\varnothing$, where $F$ is such that there is an $\alpha<\omega_1$ with
$B_\alpha\subseteq^* F$, and let
$$\mathcal I=\{ A\subseteq \G: \ \forall \alpha<\omega_1
\ \bigcup A\subseteq^* B_\alpha\}.$$

Now, the fact that $\mathcal I$ is tame follows easily by noting that
no restriction of $\mathcal I$ to a positive set is tall, and
$\mathcal I$ is inviariant as $\bigcup a\triangle I=^*\bigcup I$ for
every $I\in\mathcal I$ and $a\in\G$. The fact that $\mathcal I$ is
weakly closed follows immediately from the fact that every set of the
form $[C]^{<\omega}$ is closed in the topology of $\G$, and every set
in $\mathcal I$ is contained in en element of $\mathcal I$ of this
form ($C\in\mathcal I$ if and unly if $[\bigcup
  C]^{<\omega}\in\mathcal I$).

To see that the alternative (1) of \IIA\ fails for $\mathcal I$
note first that $C\to 0$ if and only if $C$ is a point-finite family
of finite sets and $C\in \mathcal I$. The fact that there cannot be a
countable family of elements of $\mathcal I$ intersecting infinitely
every convergent sequence follows directly from the fact that we
started with a Hausdorff gap (hence there cannot be a single such
element of $\mathcal I$) and the fact that $\mathfrak b>\omega_1$,
hence $\{B_\alpha:\alpha<\omega_1\}$ cannot form the upper half of an
$(\omega,\omega_1)$-gap by the Theorem of Rothberger.

Alternative (2) fails as $X\in \mathcal I^+ $ iff $\bigcup X\setminus
B_\alpha$ is infinite for some $\alpha<\omega_1$, and having countably
many such $X$, there is an $\alpha $ which is a witness for all of
them, hence none of them is mod $\mathcal I$ contained in
$[B_\alpha]^{<\omega}$.
\end{proof}

We shall return to further discussion of consequences of \IIA\ 
later on, but first we shall see that the axiom is consistent.

\subsection{Consistency of \IIA}

All the tools to prove consistency of the \emph{Invariant Ideal Axiom}
have been presented in \cite{hrusak-ramos-malykhin} and
\cite{brendle-hrusak}, though the models constructed there are not
models of \IIA. We shall first recall all the relevant lemmata
from the above mentioned papers. Those we can quote directly we do not
prove. Those that require some (very) minor changes we do prove, though
in all cases the changes are mere technicalities.

Recall first that the {\em Laver-Mathias-Prikry forcing}
$\L_{\mathcal{F}}$ associated to a filter $\mathcal{F}$ on
$\omega$, is defined as the set of those trees $T \subseteq
\omega^{<\omega}$ with stem $s_T$ such that for all $s \in T$
extending $s_{T}$ the set $\mbox{succ}_{T}(s)=\{n \in \omega \colon
s^{\frown}n \in T\}$ belongs to $\mathcal{F}$.  The set
$\L_{\mathcal{F}}$ is ordered by inclusion.

\smallskip
The forcing $\L_{\mathcal{F}}$ is $\sigma$-centered, and adds
generically a dominating real $\dot{\ell}_{\mathcal{F}} \colon \omega
\to \omega$ (The function $\dot{\ell}_{\mathcal{F}}$ is the unique
branch through $\omega^{<\omega}$ that belongs to all trees in the
generic filter, and it eventually dominates all ground model
reals). Its range $\dot{A}_{gen}=\text{ran}(\dot{\ell}_{\mathcal{F}})$
\emph{separates} the filter $\mathcal{F}$ (that is, the set
$\dot{A}_{gen}$ is almost contained in all members of $\mathcal{F}$
and has infinite intersection with every $\mathcal{F}$-positive set).

Names for reals in forcings of the type $\L_{\mathcal{F}}$ can
be analyzed using \emph{ranks} as introduced by Baumgartner and Dordal
in \cite{BD85}, and further developed by Brendle
\cite{Br97,Br06}. Given a formula $\varphi$ in the forcing language
and $s \in \omega^{<\omega}$, we say that $s$ {\em favors} $\varphi$
if there is no condition $T \in \L_{\mathcal{F}}$ with stem
$s$ such that $T \Vdash ``\neg \varphi"$, or equivalently, every
condition $T \in \L_{\mathcal{F}}$ with stem $s$ has an
extension $T^{\prime}$ such that $T^{\prime} \Vdash ``\varphi"$.

\smallskip
Recall also, that a forcing notion $\mathbb{P}$ {\em strongly
  preserves $\omega$-hitting} if for every sequence $\langle
\dot{A}_{n} \colon n \in \omega \rangle$ of $\mathbb{P}$-names for
infinite subsets of $\omega$ there is a sequence $\langle B_{n} \colon
n \in \omega \rangle$ of infinite subsets of $\omega$ such that for
any $B \in [\omega]^{\omega}$, if $B \cap B_{n}$ is infinite for all
$n$ then $\Vdash_{\mathbb{P}}``B \cap \dot{A}_{n}$ is infinite for all
$n"$.

In our terminology, one of the lemmas of \cite{brendle-hrusak}
becomes:

\begin{lemma}[\cite{brendle-hrusak}]\label{P:BH}
Let $\mathcal{I}$ be an ideal on $\omega$ and let
$\mathcal{F}=\mathcal{I}^{*}$ be the dual filter. Then the following
are equivalent:
\begin{enumerate}
\item $\mathcal I$ is tame,
  
\item $\L_{\mathcal{F}}$ strongly preserves $\omega$-hitting.
  
\item $\L_{\mathcal{F}}$ preserves $\omega$-hitting.
  
\end{enumerate}
\end{lemma}

and the standard preservation under finite support iteration argument
gives

\begin{lemma}[\cite{brendle-hrusak}]\label{L:Preservation}
Finite support iteration of ccc forcings strongly preserving
$\omega$-hitting strongly preserves $\omega$-hitting.
\end{lemma}

Recall also, that given an ideal $\mathcal{I}$, a forcing notion
$\mathbb{P}$, and a $\mathbb{P}$-name $\dot{A}$ for a subset of
$\omega$ we say that $\mathbb{P}$ {\em seals the ideal $\mathcal{I}$
  via $\dot{A}$ } if $\Vdash_{\mathbb{P}}``\dot{A} \in \mathcal{I}^{+}
\wedge \,\mathcal{I}\restriction \dot{A}$ is $\omega$-hitting".

Following \cite{hrusak-ramos-malykhin} we say that an ideal
$\mathcal{I}$ is {\em $\omega$-hitting mod} filter $\mathcal{F}$ if
$\mathcal{I} \cap \mathcal{F} = \varnothing$ and for every countable
family $\mathcal{H} \subset \mathcal{F}^{+}$ there is an $I \in
\mathcal{I}$ such that $H \cap I \in \mathcal{F}^{+}$ for all $H \in
\mathcal{H}$.

\begin{lemma}[\cite{hrusak-ramos-malykhin}]\label{L:Seal}
The forcing $\L_{\mathcal{F}}$ seals an ideal $\mathcal{I}$
via $\dot{A}_{gen}$ if and only if $\mathcal{I}$ is $\omega$-hitting
mod $\mathcal{F}$.
\end{lemma}

The following corollary is the main tool for the consistency of \IIA.

\begin{corollary}[\cite{hrusak-ramos-malykhin}]\label{P:MainProposition}
  Let $\G$ be a countable topological group and
  $\mathcal I\subseteq \mathcal P(\omega)$ be a tame invariant ideal
  such that both alternatives of \IIA\ fail. Then:
\begin{enumerate}

\item $\L_{\mathcal I^*}$ forces $\G$ is not groomed,
  and

\item $\L_{\mathcal I^{*}}$ strongly preserves
  $\omega$-hitting.

\end{enumerate}
\end{corollary}

\begin{proof} To see (1), one essentially only needs to translate from
  one language to another; to force that $\G$ is not groomed,
  i.e. contains a dense entangled set, means (in the language of
  sealing) adding a dense subset $\dot A$ of $\G$ such that
  $\L_{\mathcal I^*}$ seals the ideal $\mathcal{I}_{g}$for every $g\in
  \G$. Now, by Lemma \ref{L:Seal}, it is enough to show that the ideal
  $\mathcal{I}_{g}$ is $\omega$-hitting mod $\mathcal I^{*}$, and hence,
  $\L_{\mathcal I^*}$ seals the ideal $\mathcal{I}_{g}$ via
  $\dot{A}_{gen}$ for every $g\in \G$, which is simply the negation of
  the existence of a countable almost $\pi$-network of $\I$-positive sets
  (alternative (2) of \IIA). The failure of capturing of convergent
  sequences (alternative (1) of \IIA) guarantees that every element of
  $\mathcal I^*$ is dense in $\G$, hence also $\L_{\mathcal I^*}$
  forces $\dot{A}_{gen}$ to be dense in $\G$.

\smallskip

(2) follows directly from Lemma \ref{P:BH}.
\end{proof}

The last part of the forcing argument which deals with the
preservation of dense entangled sets, really uses algebra.

Let $(\G,\cdot)$ be an abstract group and let $A \subseteq\G \setminus
\{\uG\}$.  A subset $Y$ of $\G$
is called {\em $A$-large} if for every $a \in A$ and $b\in\G$, either
$b \in Y$ or $a\cdot b^{-1} \in Y$. By
$A$-$\mathsf{large}$ we will denote the collection of all subsets of
$\G$ which are $A$-large. A family $\mathcal{C}$ of subsets of
$\G$ is {\em $\omega$-hitting w.r.t. $A$} if given $\langle
Y_{n} \colon n \in \omega\rangle \subset A$-$\mathsf{large}$ there is
a $C \in \mathcal{C}$ such that $C \cap Y_{n}$ is infinite for all
$n$. Finally, We say that a relation $R \subseteq \G \times
\G$ is {\em large} if for every $a, b \in \G$, either
$\langle a,b \rangle \in R$ or $\langle a,a\cdot b^{-1}\rangle \in R$.

\begin{lemma}[\cite{hrusak-ramos-malykhin}]\label{L:group} 
Let $\G$ be a countable group and let $\mathcal I$ be a weakly closed invariant
ideal on $\G$ for which (1) of the \IIA\ fails, and let
$\langle R_{n} \colon n \in \omega \rangle$ be a sequence of large
relations. Then there is a sequence $C$ convergent to $1_\G$
such that $R_{n}^{-1}[C \setminus F] \in \mathcal I^{+}$ for every $n
\in \omega$ and $F \in [\G]^{<\omega}$.
\end{lemma}
\begin{proof}
For every $n\in \omega$ let $B_{n}=\{b \in \G \colon
R_{n}^{-1}(b) \in \mathcal I\}$ and put
$$\mathcal{S}=\{R_{n}^{-1}(b) \cdot b^{-1} \colon b \in B_{n}, n \in
\omega\} \cup \{B_{n} \colon B_{n} \in \mathcal I\}.$$

As (1) of \IIA\ fails, there is a sequence $C$ converging to
$1_\G$ such that $C \cap S$ is finite for every $S \in
\mathcal{S}$.  We claim that $R_{n}^{-1}[C \setminus F] \in \mathcal
I^+$ for every $n \in \omega$ and $F \in [\G]^{<\omega}$. To
see this, let $n \in \omega$ and $F \in [\G]^{<\omega}$ be
given. Consider two cases.

\smallskip

{\bf Case 1.} $B_{n} \in \mathcal I$.

\smallskip

Then there is an $b \in C \setminus F$ such that $R_{n}^{-1}(b) \in
\mathcal I^{+}$.

\smallskip

{\bf Case 2.} $B_{n} \in \mathcal I^{+}$.

\smallskip

Fix $b \in B_{n}$. Then $a \notin R_{n}^{-1}(b)\cdot b^{-1}$ (or,
equivalently, $a\cdot b \notin R_{n}^{-1}(b)$) for all but finitely
many $a \in C$. Since $R_{n}$ is a large relation, $\langle a\cdot b,
a \rangle \in R_{n}$ for all but finitely many $a \in C$. In
particular, $\{a \cdot b \colon a \in C\} \subseteq^{*} R_{n}^{-1}[C
  \setminus F]$ and $\{a \cdot b \colon a \in C\}$ converges to
$b$. Thus, 
$$B_{n} \subseteq R_{n}^{-1}[C \setminus F]\cup\{b\in\G:  C\cdot b\subseteq^* R_{n}^{-1}[C \setminus F]\}
,$$
hence, as $\mathcal I$ is weakly closed, also $R_{n}^{-1}[C \setminus F] \in \mathcal I^{+}$.
\end{proof}

\begin{lemma}[\cite{hrusak-ramos-malykhin}]\label{L:sealingII} 
Let $\G$ be a countable topological group and $\mathcal I$ an
invariant ideal on $\G$ for which (1) of the \IIA\ 
fails. Then
$$\Vdash_{\L_{\mathcal I^{*}}}``\mathcal{C}\text{ is
}\omega\text{-hitting w.r.t. }  \dot{A}_{gen}\text{'',}$$ where
$\mathcal{C}=\mathcal{I}_{1_{\G}}^{\perp}$ is the ideal
consisting of sequences converging to $1_{\G}$.
\end{lemma}

\begin{proof}
Aiming for a contradiction, assume that there are a sequence
$\langle \dot{B}_{n} \colon n \in \omega\rangle$ of
$\L_{\mathcal I^{*}}$-names and a condition $T^{*}\in
\L_{\mathcal I^{*}}$ such that $T^{*} \Vdash``\forall n \in
\omega \, (\dot{B}_{n} \in \dot{A}_{gen}$-$\mathsf{large}$)'', and for
every $C \in \mathcal{C}$ there are a condition $T_{C} \in
\L_{\mathcal I^{*}}$ with $T_{C} \leqslant T^{*}$, a natural
number $n_{C}$, and a finite subset $F_{C}$ of $\G$ such that
\[
T_{C} \Vdash``C \cap \dot{B}_{n_{C}} \subseteq F_{C}". \qquad (\star)
\]
For each $s \in T^{*}$ with $s \supseteq s_{T^{*}}$ and each natural
number $n$, put
$$R_{s,n}=\{\langle a,b \rangle \colon a \in \text{succ}_{T^{*}}(s)
\Rightarrow s^{\frown}a\text{ favors } b \in \dot{B}_{n}\}.$$

\medskip

{\bf Claim.} The relation $R_{s,n}$ is large.

\smallskip

Let $a$ and $b$ be two elements of $\G$. Assume that $\langle
a,b\rangle \notin R_{s,n}$. Assuming $a \in
\text{succ}_{T^{*}}(s)$ we have to show that $\langle a,a \cdot
b^{-1}\rangle \in R_{s,n}$. There is a condition $T^{\prime}\leq T^*$ with
$s_{T^{\prime}}=s^{\frown}a$ such that $T^{\prime} \Vdash ``b \notin
\dot{B}_{n}"$.    Then $T^{\prime} \Vdash``a \in \dot{A}_{gen}$ and
$b \notin \dot{B}_{n}$'', but also $T^{\prime}
\Vdash``\dot{B}_{n} \in \dot{A}_{gen}$-$\mathsf{large}$'' so
$T^{\prime} \Vdash ``a \cdot
b^{-1} \in \dot{B}_{n}"$. This finishes the proof of the claim.

\medskip

By Lemma \ref{L:group}, there is a $C \in \mathcal{C}$ such that
$R_{s,n}^{-1}[C \setminus F] \in \mathcal I^{+}$ for every $s \in
T^{*}$ with $s \supseteq s_{T^{*}}$, $n \in \omega$ and $F \in
[\G]^{<\omega}$. In particular, $R_{s_{C},n_{C}}^{-1}[C
  \setminus F_{C}] \in \mathcal I^{+}$, where $s_{C}=s_{T_{C}}$.  Pick
an $a \in \text{succ}_{T_{C}}(s_{C}) \cap R_{s_{C},n_{C}}^{-1}[C
  \setminus F_{C}]$. Then, there is a $b \in C \setminus F_{C}$ such
that $s_{C}^{\frown}a$ favors $b \in \dot{B}_{n_{C}}$, and hence there
is a condition $T\leqslant T_{C}$ whose stem extends $s_{C}^{\frown}a$
such that $T \Vdash ``b \in \dot{B}_{n_{C}}"$, a contradiction to the
initial assumption $(\star)$.
\end{proof}

We say that a forcing notion $\mathbb{P}$ {\em strongly preserves
  $\omega$-hitting w.r.t. $A$} if for every $\mathbb{P}$-name
$\dot{Y}$ for an $A$-large subset of a group $\G$ there is a
sequence $\langle Y_{n} \colon n \in \omega \rangle \subset
A$-$\mathsf{large}$ such that for any $C \subseteq \G$, if $C
\cap Y_{n}$ is infinite for all $n$ then $\Vdash_{\mathbb{P}} ``C \cap
\dot{Y}$ is infinite''. Clearly, every forcing notion that strongly
preserves $\omega$-hitting w.r.t. $A$ preserves $\omega$-hitting
w.r.t. $A$.

\begin{lemma}[\cite{hrusak-ramos-malykhin}]\label{L:Strongly-Preservation-wrt}
Let $\mathbb{P}$ be a $\sigma$-centered forcing notion. Then
$\mathbb{P}$ strongly preserves $\omega$-hitting w.r.t. $A$.
\end{lemma}

\begin{lemma}[\cite{hrusak-ramos-malykhin}]\label{L:Preservation-wrt}
Finite support iteration of ccc forcings strongly preserving
$\omega$-hitting w.r.t. $A$ strongly preserves $\omega$-hitting
w.r.t. $A$.
\end{lemma}

Now we are in position to state and prove the main theorem of this
section:

\begin{theorem}\label{riia-consistency} The Invariant Ideal Axiom \IIA\ together
  with the Martin's Axiom  $\mathsf{MA}(\sigma$-centered strongly
  $\omega$-hitting preserving$)$ is consistent with {\sf ZFC}.
\end{theorem}

\begin{proof}  Assume that the ground model $\mathbf{V}$ satisfies CH,
  split the set $S_1^2$---the stationary subset of $\omega_2$
  consisting of ordinals of cofinality $\omega_1$---into two
disjoint stationary sets $S_0$ and $S_1$, and suppose 
$\langle A_{\alpha} \colon \alpha \in S_{1}^{2}\rangle$ witnesses that both
$\diamondsuit(S_0)$, and   $\diamondsuit(S_1)$ hold\footnote{Given a
  stationary subset $S$ of $\omega_2$,
the principle $\diamondsuit(S)$ asserts the existence of a sequence
$\langle A_\alpha: \alpha \in S\rangle$ of subsets of $\omega_2$ such
that for any $A\subset \omega_2$ the set $\{\alpha\in S: A_\alpha=
A\cap\alpha\}$ is stationary.}.

Construct a finite
support iteration $\mathbb P_{\omega_2} =\langle \mathbb
P_\alpha,\dot{\Q}_\alpha :\alpha<\omega_2 \rangle$ so that at
a stage $\alpha \in S_0$, if $A_{\alpha}$ codes a $\mathbb
P_\alpha$-name for a $\sigma$-centered forcing $\hat{\Q}$
which strongly preserves $\omega$-hitting families, then let
$\Q_\alpha=\hat{\Q}$, otherwise let
$\dot{\mathbb{Q}}_{\alpha}$ be a $\mathbb{P}_{\alpha}$-name for
$\L_{\nwd^{*}(\Q)}$ where $\Q$ are the
rational numbers; at a stage $\alpha \in S_1$, if $A_{\alpha}$ codes
a group operation $\circ$ on $\omega$, a $\mathbb{P}_{\alpha}$-name
for a regular group topology $\tau$ with no isolated points on
$(\omega,\circ)$ and a $\circ$-invariant tame ideal such that
neither (1) nor (2) of the \IIA\ hold, we let
$\dot{\mathbb{Q}}_{\alpha}$ be a $\mathbb{P}_{\alpha}$-name for
$\L_{\mathcal I^{*}}$. If $\alpha$ is not of this form, let
$\dot{\mathbb{Q}}_{\alpha}$ be again $\mathbb{P}_{\alpha}$-name for
$\L_{\nwd^{*}(\Q)}$ where $\Q$ are the
rational numbers.
\smallskip
Let $G_{\omega_{2}}$ be a $\mathbb P_{\omega_2}$-generic over
$\mathbf{V}$. A standard argument shows that {\sf MA} for
$\sigma$-centered partial orders strongly preserving $\omega$-hitting
families holds in ${\bf V}[G_{\omega_2}]$.

\smallskip

We shall show that, in ${\bf V}[G_{\omega_2}]$, \IIA\ holds.

\smallskip

Aiming toward a contradiction, assume that in
$\mathbf{V}[G_{\omega_{2}}]$ there is a countable groomed group
$\G$ with a group topology $\tau$ and a tame invariant ideal
$\mathcal I$ on $\G$ satisfying neither (1) nor (2) of \IIA.

Now, by a standard closing off argument, there is a set $E \subset
S_{1}^{2}$ which is a club relative to $S_{1}^{2}$ such that for all
$\alpha \in E$,
\begin{enumerate}
\item 
${\bf V}[G_{\alpha}] \models \tau_{\alpha}$ is groomed where
  $\tau_{\alpha}=\tau \cap {\bf V}[G_{\alpha}]$,

\item ${\bf V}[G_{\alpha}] \models \mathcal I_{\alpha}=\mathcal I\cap
  {\bf V}[G_{\alpha}]$ is a $\G$-invariant tame ideal 
  satisfying neither (1) nor (2) of \IIA, and
\item every sequence in ${\bf V}[G_{\alpha}]$ which is
  $\tau_{\alpha}$-convergent is forced to be $\tau$-convergent in
  ${\bf V}[G_{\omega_2}]$.
\end{enumerate}

Therefore, at some stage $\alpha \in S_1$, we would have added a set
$A_{gen}$ such that ${\bf V}[G_{\alpha + 1}] \models A_{gen}$ is a
dense entangled subset of $\G$, i.e.  the ideal
$\restr{\mathcal{I}_{g}(\tau_{\alpha})}{A_{gen}}$ is $\omega$-hitting
for every $g\in \G$ (Proposition \ref{P:MainProposition} (1)).

We claim that $A_{gen}$ is also dense in ${\bf V}[G_{\omega_2}]$, i.e. $A_{gen}
\in \mathcal{I}_{g}^{+}(\tau)$ for every $g\in\G$.  As $\G$ is a group
it suffices to show this at $\zG$. If it were not true, in
${\bf V}[G_{\omega_2}]$, there is a $\tau$-open neighborhood $U$ of
$0$ disjoint from $A_{gen}$ such that $U \cdot U \cap
A_{gen}=\varnothing$. Then $Y=\G \setminus U$ is
$A_{gen}$-large. By Lemma \ref{L:sealingII}, in ${\bf V}[G_{\alpha +
    1}]$ the ideal
$\mathcal{I}_{0}^{\perp}(\tau_{\alpha})=\mathcal{I}^{\perp}_{0}(\tau)\cap
\mathbf V [ G_\alpha]$ is $\omega$-hitting w.r.t.  $A_{gen}$, and by
Lemmata \ref{L:Strongly-Preservation-wrt} and
\ref{L:Preservation-wrt}, it follows that the ideal
$\mathcal{I}_{0}^{\perp}(\tau_{\alpha})$ is also $\omega$-hitting
w.r.t.  $A_{gen}$ in ${\bf V}[G_{\omega_2}]$. In particular, there is
a $C \in \mathcal{I}_{0}^{\perp}(\tau_{\alpha})$ such that $C\cap Y$
is infinite, \emph{i.e.}, $C\in {\bf V}[G_\alpha]$ is a sequence
$\tau_\alpha$-converging to $0$, which intersects $Y$ infinitely
often. However, by the item 2, $\mathcal{I}_{0}^{\perp}(\tau_{\alpha})
\subset \mathcal{I}_{0}^{\perp}(\tau)$, therefore $C$ is also
$\tau$-converging to $0$. This however leads to a contradiction as
both $Y \cap U = \varnothing$ and $C\setminus U$ is finite.

 By Proposition \ref{P:MainProposition} (2) and Lemma
 \ref{L:Preservation}, in ${\bf V}[G_{\omega_2}]$ the ideal
 $\restr{\mathcal{I}_{g}(\tau_{\alpha})}{A_{gen}}$ is $\omega$-hitting
 for every $g\in\G$ contradicting that $\G$ was groomed
 in ${\bf V}[G_{\omega_2}]$.
\end{proof}

\section{Countable sequential groups under \IIA}

In this section we prove the main result of the paper, which confirms a
conjecture of the second author \cite{shibakov-nointeresting} by
proving a common extension of Theorems \ref{malykhin} and
\ref{no-inter} that provides the ultimate (consistent) classification
for the topologies of countable sequential topological groups, namely:

\begin{theorem}\label{ctble-sequential}
Assuming \IIA, every countable sequential group is either
metrizable or $k_\omega$.
\end{theorem}

Countable $\kw$ groups are completely classified
by their {\it compact scatteredness rank\/} defined as the supremum of
the Cantor-Bendixson index of their compact subspaces by the theorem
of Zelenyuk:

\begin{theorem}[E.~Zelenyuk \cite{zelenyuk}]
Countable $\kw$ groups of the same compact scatteredness rank are
homeomorphic.
\end{theorem}

To give a more concrete feel for how strong Theorem
\ref{ctble-sequential} actually is, let us introduce the following
notation: Given an indecomposable ordinal\footnote{An ordinal number
  $\alpha$ is \emph{indecomposable} if it cannot be written as an
  ordinal sum of two strictly smaller ordinals, equivalently, there is
  a $\beta\leq\alpha$ such that $\alpha=\omega^\beta$, where
  $\omega^\beta$ denotes ordinal exponentiation.}
$\alpha<\omega_1$ let $\mathcal K_\alpha$ be a fixed countable family
of compact subsets of the rationals $\Q$ closed under
translations, inverse and algebraic sums such that $\alpha=\sup
\{\rank_{CB} (K):K\in\mathcal K_\alpha\}$, and let
$$\tau_\alpha=\{U\subseteq\Q: \ \forall K\in\mathcal K_\alpha:
U\cap K\text{ is open in }K\}.$$
Then $\tau_\alpha$ is a $k_\omega$ sequential group topology on
$\Q$ and we denote $\Q_\alpha=(\mathbb
Q,\tau_\alpha)$. Note, in particular, that if $\alpha=0$ then $\tau_0$
is the discrete topology on $\Q$, and that the usual topology
on $\Q$ is similarly determined by taking into account
\emph{all} of its compact subsets, so it makes sense to denote it as
$\Q_{\omega_1}$. Hence Theorem \ref{ctble-sequential} can be
reformulated as:

\begin{theorem}
Assuming \IIA, for every infinite countable sequential group
$\G$ there is exactly one $\alpha\leq \omega_1$ such that
$\G$ is homeomorphic to $\Q_{\omega^\alpha}$.
\end{theorem}

Note that in the argument above we may have started with an arbitrary
countable \emph{topologizable} (i.e.\ admitting a nondiscrete Hausdorff group
topology) group $\G$ instead of $\Q$ by possibly choosing a coarser
first countable topology on $\G$ first. Thus \emph{every}
countable topologizable group admits every possible $k_\omega$ group
topology showing that in a model of \IIA\ the algebraic structure of
the group has almost no influence on the kind of sequential topology
the group can admit. Indeed, in such models the number of
nonisomorphic topologizable countable groups ($\C$) is
greater than the number of nonhomeomorphic sequential group topologies ($\omega_1$).

\smallskip

We shall prove Theorem \ref{ctble-sequential} in a sequence of
lemmata.

\begin{lemma}\label{scattered.vD}Let $\G$ be a countable nondiscrete
sequential group. Suppose ${\mathcal P}\subseteq\csct(\G)$ is a countable
family such that for every $S\to\uG $ there exists a
$P\in{\mathcal P}$ such that $|S\cap P|=\omega$. Let ${\mathcal
D}\subseteq[\G]^\omega$ be a countable family of closed discrete
subsets of $\G$. Then for every $g\in \G$ there exists an open $U\ni g$
such that $U\cap D$ is finite for every $D\in{\mathcal D}$.
\end{lemma}

\begin{proof}
Let ${\mathcal D}=\set D_n:n\in\omega.\subseteq2^\G$ be a collection of
closed discrete subsets of $\G$. For brevity, call a point
$g\in \G$ a {\em\vD-point\/} of ${\mathcal D}$ if for every open $U\ni g$
there is a $D\in{\mathcal D}$ with the property $|U\cap D|=\omega$. The
statement of the lemma is equivalent to claiming that there are
no \vD-points of ${\mathcal D}$. Suppose $g\in \G$ is a \vD-point of ${\mathcal
D}$. By translating each $D\in{\mathcal D}$ if necessary, we may assume
that $g=\uG $.

Let ${\mathcal P}=\set P_n:n\in\omega.$ be a collection of closed scattered
subsets of $\G$ such that for any $S\to\uG $ there is a $P\in{\mathcal P}$
such that $|S\cap P|=\omega$. By requiring ${\mathcal P}$ to be closed
under finite unions we may assume that $S\ain P$.

Pick a family $\set O_n:n\in\omega.$ of open neighborhoods of $\uG $ such
that $\cl{O_{n+1}}\subseteq O_n$ and $\bigcap_{n\in\omega}O_n=\{\uG \}$. Put
$P=\bigcup_{n\in\omega}P_n\cap\cl{O_n}$. One may verify that $P$ is
closed, scattered, and for any $S\to\uG $, $S\ain P$. Changing $P$ if
necessary, we may further require that $\alpha_P=\scl(\uG ,P)$ is the smallest one
among all $P$ with such properties.

Let $P'=\set p\in P\setminus\{\uG \}:\scl(p,P)\geq\alpha_P.$. Note that
$\uG \not\in\cl{P'}$ and we may therefore assume (by taking an
appropriate subset of $P$, if necessary) that
$\scl(\uG ,P)=\scl(P)>\scl(p,P)$ for any $p\in P$ such that $p\not=\uG $. Now

\begin{couup}[series=ipr]

\item\label{one.absorber}$P$ is a closed scattered subset of $\G$ such
that $S\ain P$ for every $S\to\uG $; moreover,
$\alpha_P=\scl(P)=\scl(\uG ,P)>\scl(p, P)$ for any $p\in
P\setminus\{\uG \}$, and $\alpha_P$ is the smallest possible.

\end{couup}

Let $g\in \G$ and consider the family ${\mathcal D}_g=\set
D_n\cap(g\cdot P):n\in\omega.$. Suppose $g$ is not a \vD-point of ${\mathcal D}_g$ for any
$g\in \G$. For each $g_i\in \G$ pick an open neighborhood $U_i\ni g_i$ such
that $U_i\cap(g_i\cdot P)\cap D_n=F_i^n$ is finite for every
$n\in\omega$. Put $D_n'=D_n\setminus\bigcup_{i\leq n}F_i^n$. Then $\uG $ is
a \vD-point of ${\mathcal D}'=\set D_n':n\in\omega.$, therefore,
$\uG \in\cl{\bigcup{\mathcal D}'}$.

Let $S\subseteq\bigcup{\mathcal D}'$ be an infinite sequence such that $S\to
g_i$ for some $g_i\in \G$. Then, by~\ref{one.absorber} $S\ain
U_i\cap(g_i\cdot P)$. Since every $D\in{\mathcal D}'$ is closed discrete,
we may assume that $S=\seq s_n$ is such that $s_n\in D_{m(n)}'$ for some
$m(n)\geq n$. Let $n>i$ be such that $s_n\in U_i\cap(g_i\cdot P)$. Then
$s_n\in F_i^{m(n)}$, $i<n<m(n)$ contradicting $s_n\in
D_{m(n)}'=D_{m(n)}\setminus\bigcup_{i\leq m(n)}F_i^{m(n)}$. Thus no such
$S$ exists  making $\bigcup{\mathcal D}'$ almost disjoint from every
convergent sequence in $\G$ contradicting $\uG \in\cl{\bigcup{\mathcal
D}'}$ and the sequentiality of $\G$.
%assumption made at the beginning of the proof.

We may therefore assume that $D\subseteq P$ for every $D\in{\mathcal
D}$ and some $g\in \G$ is a \vD-point of ${\mathcal D}$. Note that $g\in P$,
since $P$ is closed. Let $p\in P$ be a \vD-point of ${\mathcal D}$ such
that $\scl(p,P)$ is the smallest. By picking a neighborhood
$U\ni p$ relatively open in $P$ such that $\scl(x,P)<\scl(p,P)$ for any
$x\in\cl{U}\setminus\{p\}$ and restricting ${\mathcal D}$ to $U$, if
necessary, we may assume that $p$ is the only \vD-point of ${\mathcal
D}$. Using an argument similar to the one in the previous paragraph,
by possibly removing a finite subset from each $D\in{\mathcal D}$ we may
assume that ${\mathbb D}=\bigcup{\mathcal D}\cup\{p\}$ is closed, $p$
is the only nonisolated point of $\mathbb D$, and $p$ is a \vD-point
of ${\mathcal D}$.

Consider the translation ${\mathcal D}''=\set D\cdot p^{-1}:D\in{\mathcal D}.$ of ${\mathcal
D}$. Suppose $\uG \in\cl{\bigcup{\mathcal D}''\setminus P}$. By the closedness of
$\mathbb D$ and property~\ref{one.absorber} the set $\bigcup{\mathcal D}''\setminus P$
contains no infinite converging sequence, contradicting the
sequentiality of $\G$.
%assumption made at the beginning of the proof. We thus obtain

\begin{couup}[ipr]\label{vD.space}

\item There exists a countable family $\mathcal D$ of closed discrete
subsets of $\G$ such that $\bigcup{\mathcal D}\subseteq P$, $\uG $ is the only
nonisolated point of $\bigcup{\mathcal D}\cup\{\uG \}$, which is closed in $\G$,
and $\uG $ is a \vD-point of ${\mathcal D}$.

\end{couup}

Suppose there exists an $S\to\uG $ such that $(S\cdot p)\setminus P$ is
infinite for every $p\in P\setminus\{\uG \}$. We may then pick an infinite sequence
$S'\subseteq S\cdot S$ so that $S'\to\uG $ and $S'\subseteq \G\setminus
P$, contradicting~\ref{one.absorber}.
Thus for every sequence $S\to\uG $ there exists a $p\in P\setminus\{\uG \}$
such that $(S\cdot p)\ain P$.

Suppose $A\subseteq\G$ is such that for some ordinal $\beta$,
$\scl(a,P)\leq\beta<\alpha_P$ for every $a\in A\cap P$. Then there
exists a sequence $S\to\uG $ such that $S\setminus\bigcup_{a\in F}(P\cdot
a^{-1})$ is infinite for every $F\in[A]^{<\omega}$.

Indeed, suppose no such $S$ exists and let $A=\set p_n:n\in\omega.$ list all the
points in $A$. For each $n\in\omega$ find a neighborhood $U_n\ni p_n$
relatively open in $P$ so that $\scl(\cl{U_n}\cap P)\leq\beta$ if
$p_n\in P$ and put $P_n'=(\cl{U_n}\cap P)\cdot p_n^{-1}$. If
$p_n\not\in P$, put $P_n'=\varnothing$. Note that ${\mathcal P}'=\set
P_n':n\in\omega.$ is a collection of closed scattered subsets of $\G$
with the property that $\scl(P_n)\leq\beta$ for every $n\in\omega$,
and for every $S\to\uG $ there exists an $F\in[\omega]^{<\omega}$ such that $S\ain
\bigcup_{n\in F}P_n'$.

Repeating the construction used to build $P$ out of $P_n$ at the
beginning of this argument we may construct a closed scattered
$P'\subseteq \G$ such that $\scl(P')\leq\beta$ and $S\ain P'$ for
every $S\to\uG $ contradicting the minimality of $\alpha_P$
in~\ref{one.absorber}.

Let $P\setminus\{\uG \}=\set p_n:n\in\omega.$ list all the points in $P$
other than $\uG $. For each $n\in\omega$ pick an open $O_n\ni\uG $ such that
$\beta=\scl(p_n, P)=\scl(\cl{O_n\cdot p_n}\cap P)<\alpha_P$, $\cl{O_{n+1}}\subseteq O_n$,
and $\bigcap_{n\in\omega}O_n=\{\uG \}$.

Restricting ${\mathcal D}$ to $O_0$ if necessary, assume that $\bigcup{\mathcal
D}\subseteq O_0$. By induction, pick disjoint closed discrete
$D_n'\subseteq\cup{\mathcal D}$ so that $D_n'\subseteq O_n$ and each $D_n$
is covered by finitely many $D_n'$. To see that this is
possible, put $D_n'=(D_n\cap O_n)\cup((O_n\setminus
O_{n+1})\cap(\bigcup{\mathcal D}\setminus\bigcup_{i<n}D_i'))$ and observe that
the intersection inside the second pair of parentheses is a closed and
discrete subspace of $\G$, since
$\uG$ is the only nonisolated point of $\mathbb D$. Put ${\mathcal
D}'=\set D_n':n\in\omega.$. Note that $\uG $ is a \vD-point of ${\mathcal
D}'$. To simplify notation, we will assume that $D_n\subseteq O_n$ in
what follows.

Let $n\in\omega$. By the choice of $O_n$,
$D_n\subseteq O_n\subseteq O_i$, so $D_n\cdot p_i\subseteq O_i\cdot
p_i$, whenever $i\leq n$. Thus there is a $\beta<\alpha_P$ such that
$\scl(a,P)\leq\beta$ for every $a\in
A_n=\bigcup_{i\leq n}D_n\cdot p_i$. Find a sequence $S_n\to\uG $ such
that $S_n\setminus\bigcup_{a\in F}(P\cdot
a^{-1})$ is infinite for every $F\in[A_n]^{<\omega}$. Let $D_n=\set d_i:i\in\omega.$ and
$S_n=\set s_i:i\in\omega.$ be 1-1 listings of $D_n$ and $S_n$.
For each $i\in\omega$ pick an
$m(i)>n$ so that $m(i)$ is strictly increasing, $s_{m(i)}\cdot
d_i\cdot p_j\not\in P$ for every $i\in\omega$ and $j\leq n$, and
$e_i^n=s_{m(i)}\cdot d_i\in O_n$. Note that the latter is possible, since
$S\to\uG$ and $d_i\in D_n\subseteq O_n$. Put $B_n=\set
e_i^n:i\in\omega.$, $B=\bigcup_{n\in\omega}B_n$ and note that each $B_n$ is a
closed and discrete subspace of $\G$.

Now, $\uG \in\cl{B}$. Indeed, let $U\ni\uG $ be any open neighborhood of
$\uG $. Find an open $V\ni\uG $ such that $V\cdot V\subseteq U$. Then $V\cap D_n$
is infinite for some $n\in\omega$, since $\uG $ is a \vD-point of ${\mathcal
D}$. Also, $S_n\ain V$. Thus for some large enough $i\in\omega$ both
$d_i\in V$ and $s_{m(i)}\in V$ showing that $e_i^n\in U$. 

Let $C\subseteq B$ be an infinite sequence such that $C\to g$ for some
$g\in \G$. Since each $B_n$ is closed and discrete,
we may assume that $C=\set e_{i(k)}^{n(k)}:k\in\omega.$ where $n(k)$
is strictly increasing. Since $e_i^n\in O_n$, $C\to\uG $. Thus, there
exists a $j\in\omega$ such that $C\cdot p_j\ain P$. Pick a $k\in\omega$
large enough so that $n(k)>j$ and $e_{i(k)}^{n(k)}\cdot p_j\in P$. At the same time
$e_{i(k)}^{n(k)}\cdot p_j=s_{m(i(k))}\cdot d_{i(k)}\cdot p_j\not\in P$ by the choice
of $s_{m(i)}$, a contradiction.
\end{proof}

\begin{lemma}\label{nwd.pi}
Let $\G$ be a countable nondiscrete sequential group. Suppose ${\mathcal
P}\subseteq\nwd(\G)$ is a countable family such that
for every $S\to\uG $ there exists a $P\in{\mathcal P}$ such that $|S\cap
P|=\omega$. Then $\G$ does not have a countable $\pi$-network at $\uG $
that consists of dense in themselves sets.
\end{lemma}

\begin{proof}
Note that $\G$ is not Fr\'echet by~Lemma~\ref{nwd-lemma}, and thus does
not contain a closed subspace homeomorphic to ${\mathbb D}(\omega)$ by
Proposition~\ref{ts.properties}.

Let ${\mathcal D}=\set D_n:n\in\omega.$ be a $\pi$-network at $\uG $ such
that each $D\in{\mathcal D}$ is dense in itself. By translating each
element of ${\mathcal D}$ if necessary, we may assume that $\uG \in D$ for
every $D\in{\mathcal D}$.

Fix open $O_n\ni\uG $ so that $\cl{O_{n+1}}\subseteq O_n$ and
$\bigcap_{n\in\omega}O_n=\{\uG \}$.

Let ${\mathcal P}=\set P_n:n\in\omega.\subseteq\nwd(\G)$ be such that for
every $S\to\uG $ there exists a $P\in{\mathcal P}$ such that $|S\cap
P|=\omega$. Just as in the proof of Lemma~\ref{scattered.vD} we may
construct a $P\in\nwd(\G)$ such that for every $S\to\uG $, $S\ain P$. By taking the
closure of $P$ if necessary, we may assume that $P$ is closed.

Let $g\in \G$. Define $d(g)=\so(g,\G\setminus P)$. Let
$\alpha_P=d(\uG )$. Note that $\alpha_P>1$ by the choice of $P$.
We proceed to prove the following claim by induction on $\alpha$.

\begin{couup}[ipr]

\item\label{scaffold}Let $p\in P$ and $d(p)=\alpha$ for some
$\alpha<\omega_1$. There exists a $T\subseteq \G\setminus P$ and a
neighborhood assignment $W:T\to\tau(\G)$ such that the following
properties hold:
\begin{couup}[label={\rm(\alph*)},ref={\rm(\arabic{couupi}\alph*)}]
\item\label{scaffold.free}$p\in[T]_\alpha$, if $p'\in\cl{T}$ then
$d(p')=\so(p',T)$;

\item\label{scaffold.stratified}$g\in W(g)\setminus P$ for every $g\in T$,
the $W(g)$ are disjoint; if $s_i\in W(g_i)\setminus P$ is such
that $s_i\to g$ for some $g\in \G$ and all $g_i$ are distinct then
$g_i\to g$;
\end{couup}

\end{couup}

Let $d(p)=\alpha$ for some $p\in P$. If $\alpha=1$ there exists an
infinite sequence of $g_i\in\G\setminus P$ such that $g_i\to p$. Thinning
out the sequence and reindexing, if necessary, pick disjoint open $W(g_i)\ni g_i$ so
that $W(g_i)\subseteq O_i\cdot p$. Put $T=\set
g_i:i\in\omega.$. Properties~\ref{scaffold.free}
and~\ref{scaffold.stratified} are easy to check.

Thus we can assume $\alpha>1$. Let $p^n\to p$ and $\alpha_n<\alpha$
be such that $p^n\in O_n\cdot p$ and 
$p^n\in[\G\setminus P]_{\alpha_n}$ for every $n\in\omega$. Since
$d(p^n)\leq\alpha_n<\alpha$, by the induction hypothesis there exist
$T_n\subseteq \G\setminus P$ and $W_n:T_n\to\tau(\G)$ that
satisfy~\ref{scaffold}. Pick a sequence of open disjoint $V_n\ni p^n$
such that $V_n\subseteq O_n\cdot p$ after thinning out an reindexing
if necessary.
By passing to subsets and reindexing again, if
necessary, we may assume that the $\cl{T_n}$ are disjoint, $T_n\subseteq
O_n\cdot p\cap V_n$, and $W_n(g)\subseteq O_n\cdot p\cap V_n$ for every $n\in\omega$ and
$g\in T_n$. Let $T=\bigcup_{n\in\omega}T_n$ and define $W:T\to\tau(\G)$
by $W(g)=W_n(g)$ whenever $g\in T_n$.

By the choice of $T_n$ and $O_n$,
$\cl{T}=\{p\}\cup\bigcup_{n\in\omega}\cl{T_n}$. If $p'\in\cl{T_n}$ then
$d(p')=\so(p', T_n)=\so(p', T)$ by the inductive hypothesis and the
choice of $T_n$. Since $d(p)=(\sup_n\alpha_n)+1$ and
$d(p^n)\leq\alpha_n$, $d(p)=\so(p,T)$.

Let $s_i\in W(g_i)\setminus P$ for some $g_i\in T$ be such that $s_i\to g$. By
thinning out and reindexing we may assume that either $g_i\in T_n$ for
some fixed $n\in\omega$ or $g_i\in T_{n(i)}\subseteq O_{n(i)}\cdot p$ for some
strictly increasing $n(i)$. In the first case $g_i\to g\in P$ by the
choice of $T_n$. Otherwise $s_i\in W_{n(i)}(g_i)\subseteq
O_{n(i)}\cdot p$ so $s_i\to p$ by the choice of $O_n$ contradicting
$s_i\in \G\setminus P$ and $d(p)=\alpha>1$.

Pick $T$ and $W$ that satisfy~\ref{scaffold} for $p=\uG $ and let $T=\set
t_n:n\in\omega.$ be a 1-1 enumeration of the points of $T$.
Pick $U_n\subseteq (W(t_n)\setminus P)\cdot t_n^{-1}$ so that $U_{n+1}\subseteq U_n$,
$\bigcap_{n\in\omega}U_n=\{\uG \}$, and $\uG\in\cl{\set t_k:U_k\cdot
t_k\subseteq O_n.}$ for every $n\in\omega$. Note that each $U_n\cdot
t_n\subseteq\G\setminus P$. Let
$k\in\omega$ and show that

\begin{couup}[ipr]

\item\label{graded.discrete}there exists a closed discrete subset
$E_k\subseteq D_k$ such that $E_k=\set e_n^k:n\in\omega.$ and
$e_n^k\in U_n$ for every $n\in\omega$;

\end{couup}

If no such $E_k$ exists, the $U_n\cap D_k$ form a countable base
of neighborhoods of $\uG $ in $D_k$. Since $\uG \in
D_k$, and each $D_k$ is dense in itself, this implies the existence of
a closed copy of $\mathbb D(\omega)$ in $\G$ contradicting the
sequentiality of $\G$ as noted at the beginning of this proof.

Consider the set $D^k=\set d^n:d^n=e_n^k\cdot t_n, U_n\cdot
t_n\subseteq O_k, n\in\omega.$. Then $D^k\subseteq O_k\setminus P$ for every
$k\in\omega$ and $d^n\in
W(t_n)$ for every $n\in\omega$ by $e_n^k\in U_n$ and the choice of
$U_n$. Suppose $d^{n(i)}\to d$ for some $d\in
\G$. By~\ref{scaffold.stratified} $t_{n(i)}\to d$ so
$e_{n(i)}^k=d^{n(i)}\cdot t_{n(i)}^{-1}\to\uG $ contradicting the choice of
$e_n^k$. Thus each $D^k\subseteq O_k$ is closed and discrete in $\G$.

Suppose $S\to g$ for some $g\in \G$ is an infinite sequence such that
$S\subseteq\bigcup_{k\in\omega}D^k$. Since each $D^k$ is closed discrete,
we may assume that $S=\set s_n:n\in\omega.$ where $s_n\in D^{k(n)}$
for some strictly increasing $k(n)$. Then $s_n\in O_{k(n)}$ and
$S\to\uG $ contradicting $S\subseteq \G\setminus P$ and the choice of $P$.

Let $U\ni\uG $ be open and find an open $V\ni\uG $ such that $V\cdot V\subseteq
U$. Let $k\in\omega$ be such that $D_k\subseteq V$, and let $t_n$ be
such that $U_n\cdot t_n\subseteq O_k$ and $t_n\in V$.
Then $d^n=e_n^k\cdot t_n\in D^k\cap V\cdot V$. Thus
$\uG \in\cl{\bigcup_{k\in\omega}D^k}$ contradicting the sequentiality of $\G$.
\end{proof}

\begin{lemma}\label{split.cover}
Let $X$ be a sequential space, $S\subseteq X$, and
$x\in\cl{S}$ for some $x\in X$. Let $\I\subseteq2^X$ be a cover of
$S$. Then either there exists an $I\in\I$ such that $x\in\cl{S\cap I}$
or there is a countable $\I^*\subseteq[\I]^\omega$ such that whenever
$\I'\subseteq\I$ is such that $\I'\cap\I''$ is infinite for every $\I''\in\I^*$,
$x\in\cl{S\cap\bigcup\I'}$. 
\end{lemma}
\begin{proof}
The proof proceeds by induction on $\so(x,S)$. The case
$\so(x,S)=0$ is trivial so assume $\so(x,S)=\alpha+1$ and the lemma is
proved for all successor $\beta\leq\alpha$. Pick a sequence
$T\subseteq X$ such that $T\to x$ and $\so(y,S)=\beta_y\leq\alpha$ for
every $y\in T$ and consider the two alternatives that follow from the
inductive hypothesis. 

First, suppose the set $T'=\set y:y\in\cl{I_y\cap S}\text{ for some
}I_y\in\I.$ is infinite. If the family $\I'=\set I_y:y\in T'.$ is finite
then there is an $I\in\I'$ such that $x\in\cl{S\cap I}$. Otherwise put
$\I^*=\{\I'\}$.

Alternatively, assume for every $y\in T$ there is a countable
$\I^*_y\subseteq[\I]^\omega$ such that $y\in\cl{S\cap\bigcup\I'}$ for
any $\I'$ such that $\I'\cap\I''$ is infinite for every
$\I''\in\I^*_y$. Put $\I^*=\bigcup_{y\in T}\I^*_y$.
\end{proof}

\begin{lemma}\label{pretame}
Let $X$ be a countable sequential space, $\I\subseteq2^X$
be an ideal with the following properties: $\I$ contains all
singletons, $\cl{I}\in\I$ for every $I\in\I$,
and whenever $A\in\I^+$ there is a $Y\subseteq A$, $Y\in\I^+$ such
that $\cl{Y\setminus I}=\cl{Y}$ for any $I\in\I$. Then $\I$ is tame.
\end{lemma}
\begin{proof}
Suppose $\I$ is not tame and let $A\in\I^+$, $f:A\to\omega$ witness
this. Using the property of $\I$ from the statement of the lemma, find
$Y\subseteq A$, $Y\in\I^+$ such that $\cl{Y\setminus I}=\cl{Y}$ for any
$I\in\I$. Let $y\in Y$. Since $\I$ contains $\{y\}$ and $X$ is
sequential, there exists a sequence $T\subseteq \cl{Y}\setminus\{y\}$ such
that $T\to y$. Let $T=\seq y_i$, $I_i=f^{-1}(i)\in\I$, and for
every $i\in\omega$ pick a subset $S_i\subseteq
Y\setminus\bigcup_{j<i}I_j$ such that $y_i\in\cl{S_i}\not\ni y$.

Note that $\I_f=\set I_i:i\in\omega.$ is a cover of
$S=\bigcup_{i\in\omega}S_i$ and $y\in\cl{S}$. Since $I_i\cap
S\subseteq\bigcup_{j\leq i}S_i$ and $y\not\in\cl{\bigcup_{j\leq
i}S_i}$, the first alternative of Lemma~\ref{split.cover} fails so
there exists a countable $\I^*_y\subseteq[\I_f]^{\omega}$ such that
$y\in\cl{S\cap\bigcup\I'}$ for any $\I'\subseteq\I$ with the property
that $\I'\cap\I''$ is infinite for every $\I''\in\I^*_y$.

Let ${\mathcal J}=\set f[\bigcup\I'']:\I''\in\I^*_y,
y\in Y.\subseteq2^\omega$. Since $f[\restr{\I}{A}]$ is
$\omega$-hitting, there exists a $J\in f[\restr{\I}{A}]$ such that
$J\cap J'$ is infinite for every $J'\in{\mathcal J}$. Let $\I'=\set
f^{-1}(n):n\in J.$. Then $\I'\cap\I''$ is infinite for every
$\I''\in\I^*_y$ so $y\in\cl{f^{-1}[J]}$ for every $y\in Y$. Thus
$Y\subseteq\cl{f^{-1}[J]}\in\I$ contradicting $Y\in\I^+$.
\end{proof}

\begin{lemma}\label{tame}
Let $\G$ be a countable nondiscrete sequential group. Then each of
$\nwd(\G)$, $\cpt(\G)$, and $\csct(\G)$ is tame.
\end{lemma}

\begin{proof}
For $\cpt(\G)$ the statement can be proved directly. For $\nwd(\G)$
and $\csct(\G)$ it is sufficient to establish the properties listed in
Lemma~\ref{pretame}. In the case of $\nwd(\G)$ one may pick
$Y=\Int(\cl{A})\cap A$, while for $\csct(\G)$ the choice of the full
Cantor-Bendixson derivative of $A$ as $Y$ satisfies the conditions of
Lemma~\ref{pretame}. 
\end{proof}

\begin{lemma}\label{iaa.chain}
Let $\G$ be a countable, sequential non metrizable, non $\kw$
group. Then one of the $\nwd(\G)$, $\cpt(\G)$, or $\csct(\G)$ is a tame
invariant ideal that satisfies neither~\ref{seq.capture}
nor~\ref{almost.pi} of the \IIA.
\end{lemma}

\begin{proof}
That each of the ideals is tame follows from Lemma~\ref{tame}. The
invariance is trivial.

Since $\nwd(\G)$ never satisfies~\ref{almost.pi} for a nonmetrizable
$\G$ (see~\cite{hrusak-ramos-malykhin}, Proposition~5.2), we may assume that $\nwd(\G)$
satisfies~\ref{seq.capture}. Suppose $\csct(\G)$
satisfies~\ref{almost.pi}. Then there exists a countable ${\mathcal
  D}\subseteq\csct^+(\G)$ such that for any nonempty open $U\subseteq \G$
there exists a $D\in{\mathcal D}$ with the property $D\setminus
U\in\csct(\G)$. By replacing each $D$ with a full
Cantor-Bendixson derivative of itself we may assume that each $D$ is
dense in itself. Applying Lemma~\ref{nwd.pi} we arrive at a
contradiction. Thus $\csct(\G)$ does not satisfy~\ref{almost.pi}.

Suppose $\csct(\G)$ satisfies~\ref{seq.capture}. If $\cpt(\G)$
satisfies~\ref{almost.pi} there exists a countable family ${\mathcal
  D}$ of closed, noncompact subsets of $\G$ such that for any open
$U\subseteq \G$ there exists a $D\in{\mathcal D}$ such that $D\setminus
U$ is compact. By picking an infinite closed discrete subset in each
$D\in{\mathcal D}$ and applying Lemma~\ref{scattered.vD} we arrive at
a contradiction. Thus either $\csct(\G)$ does not
satisfy~\ref{seq.capture} or $\cpt(\G)$ does not
satisfy~\ref{almost.pi}.

Since $\G$ is not $\kw$, $\cpt(\G)$ cannot satisfy~\ref{seq.capture}.
\end{proof}

This concludes the proof of Theorem~\ref{ctble-sequential}. The result
has the following corollary, which illuminates the behavior of
sequential groups under taking products (part~(\ref{Fr.prod}) is an
obvious corollary of Theorem~\ref{just.Frechet} and has been included for completeness):

\begin{corollary}\label{products}
  Assume \IIA.
\begin{enumerate} 
    \item\label{Fr.prod} The product of at most countably many separable Fr\'echet
      groups is Fr\'echet, and
    \item The product of finitely many countable sequential groups
      which are either discrete or not Fr\'echet is sequential.
\end{enumerate}
\end{corollary}

\begin{proof}
It suffices to note that
\begin{enumerate}
\item $\Q_{\omega^\alpha}\times\Q_{\omega^\beta}\simeq
  \Q_{\omega^\beta}$ if $\alpha<\beta<\omega_1$,
\item $\Q_{\omega^\alpha}\times\Q_{\omega_1}$ is not
  sequential if $0<\alpha<\omega_1$,
\item $\Q_{0}\times\Q_{0}\simeq \Q_{0}$, and
\item $\Q_{0}\times\Q_{\omega_1}\simeq
  \Q_{\omega_1}\times\Q_{\omega_1}\simeq\Q_{\omega_1}$.
\end{enumerate}
\end{proof}

We do not know at the moment whether it is consistent that the product
of two sequential groups which are not Fr\'echet is sequential
(independently of their cardinality).

\section{Examples, concluding remarks, and open questions}
The example below can probably be constructed using the techniques
of~\cite{tkachenko87} but we chose to provide a direct proof. An
appeal to~\cite{tkachenko87} would require a proof of the normality
of finite powers of $\gamma\N$ spaces, as well as an
adaptation of the free topological group arguments
from~\cite{tkachenko87} to the free boolean group construction used
here.

The nontrivial case of the proof below assumes ${\bf t}=\omega_1$,
however, the statement of the example is meant to emphasize the
fact that one of the two `pathologies' exists in every model of ZFC:
either there is a separable nonmetrizable Fr\'echet group or a
(possibly uncountable) sequential group that is not $\kw$. The authors
do not know whether any separable locally compact first countable countably
compact non compact space may be used in place of $\gamma\N$. 

\begin{example}\label{cw.not.kw}
  If there is no separable nonmetrizable Fr\'echet group then
  there exists a separable sequential $c_\omega$ group $\G$ that is not $\kw$.
\end{example}
\begin{proof}
Since ${\bf t}>\omega_1$ implies the existence of a separable
nonmetrizable Fr\'echet group (see, for
example~\cite{hrusak-ramos-precompact}) we may assume that ${\bf
t}=\omega_1$.
 
Let $X=D\cup\omega_1$ be a countably compact $\gamma\N$ space
(see~\cite{nyikos90}, Example~2.2)
where $D$ is the set of isolated points, disjoint from $\omega_1$,
which has the usual topology.  Let $X\cup\{\infty\}$
be the one-point compactification of $X$, which is also a subspace of some
boolean group $H$ that is (algebraically) generated by
$X\cup\{\infty\}$. We shall assume that $X\cup\{\infty\}$ is
linearly independent over $H$ (in particular this means $0_H\not\in
X\cup\{\infty\}$). The free boolean group over 
$X\cup\{\infty\}$ (see \cite{sipacheva-bool}) would have all the desired
properties (in fact it is not difficult to show that any group
satisfying the properties above is naturally isomorphic to the free
boolean group over $X\cup\{\infty\}$). Below we use the convention
that the elements of such a group are finite sets of elements of
$X\cup\{\infty\}$ with the symmetric difference as the group
operation.

Let $\G$ be the subgroup of $H$ generated by $X$ and let $A\subseteq \G$
be such that $\zG\in\cl{A}\setminus A$ (here the closure is taken in
the topology induced by $H$). We must show that there exists a
sequence $S\subseteq A$ such that $S\to x\in \G\setminus A$. Suppose no
such sequence exists.

Let $n\in\omega$ be the smallest number such that 
$\zG\in\cl{A\cap\sum^nX}$. Note that such an $n$ exists by the
definition of the topology on $H$. We will assume that
$A\subseteq\sum^nX$. By the minimality of $n$ (truncating $A$ if
necessary) we may assume that $|a|=n$ for every $a\in A$. Write an
arbitrary $a\in A$ as $a=d+w$ where $d\in\la{D}$ and
$w\in\la{\omega_1}$ and put $\delta(a)=|d|$. Note that for
every $k\leq n$ the set $A_k=\set a\in A:\delta(a)\leq k.$ is a
sequentially closed subset of $A$.

Let $k\leq n$ be the smallest such that $\zG\in\cl{A_k}$. Replacing $A$
with $A_k$ and using the minimality of $k$ we may assume (again
truncating $A$ if necessary) that $\delta(a)=k$ for every $a\in
A$. Let $D_A=\bigcup\set t\in D:t\in a\in A.$
and suppose $D_A$ is infinite. Note that $\sum^nX$ is sequentially
compact. Using this and the property of $D_A$ we may pick a convergent
sequence $S\subseteq A$ such that $S\to x$ for some $x\in\sum^nX$ and
for each $s\in S$ there is a $t_s\in s\cap D$ such that $\set t_s:s\in
S.\to\alpha\in\omega_1$. Then $\delta(x)<k$ so $x\not\in A$.

Thus we may assume that $D_A$ is finite. Note that this implies that
$D_A$ is empty (otherwise $\zG\not\in\cl{A}$). Therefore
$A\in\sum^n\omega_1$. Define $l\leq n$ to be the largest with the
following property: for any $\alpha<\omega_1$ there exists an $a\in A$
such that $|a\setminus\alpha|\geq l$. Note that we may
assume that $l\geq1$, otherwise $A$ is countable with a metrizable
closure.

Suppose $l\geq2$. Recursively pick a sequence $a_i\in A$ such that for
some distinct $\alpha_i,\beta_i\in a_i$,
$\alpha_{i+1},\beta_{i+1}>\max\{\alpha_i,\beta_i\}$. By passing to a
subsequence if necessary we may assume that $a_i\to a\in\sum^n\omega_1$. Since
$\alpha_i,\beta_i\to\gamma$ for some $\gamma\in\omega_1$, $|a|<n$,
showing that $a\not\in A$. We may thus assume that $l=1$.

This implies the existence of an $\alpha\in\omega_1$ such that every
$a\in A$ can be written as $a=\{\beta_a\}+b_a$ where $\beta_a>\alpha+1$ and
$b_a\in[\alpha+1]^{n-1}$. If $n>1$, the set $A'=\set b_a:a\in A.$ is a
sequentially closed (and therefore compact) subset of $H$ such that $\zG\not\in A'$ (otherwise
$A\cap\omega_1\not=\varnothing$ contradicting the choice of
$A$ and $n>1$). Now $U=H\setminus
(A'+(\set\beta:\beta>\alpha+1.\cup\{\infty\}))$ is an open
neighborhood of $\zG$ such that $U\cap A=\varnothing$ contradicting the
choice of $A$. Hence $n=1$ implying $\zG\not\in\cl{A}$, a contradiction.
\end{proof}

The statement of \IIA\ given at the beginning of this paper may appear
somewhat technical in that it lists several restrictions on both the
space (groomed), as well as the ideal (tame, invariant, weakly
closed). This complexity may be significantly reduced in most
applications, however. Most natural ideals (including all used in this
paper) in sequential spaces are generated by their (sequentially)
closed members while tameness can be replaced by the topological condition
defined in Lemma~\ref{pretame}, namely the existence of a `kernel' in
each positive set. One may prefer a weaker version of \IIA\ that
states that for every invariant ideal generated by sequentially closed
sets for which the conditions in Lemma~\ref{pretame} are satisfied, one of the
two alternatives in the statement of \IIA\ holds.

Limiting the class of spaces may also make applications of \IIA\ more
transparent. Call the following satement the \emph{Unrestricted Ideal Axiom}
or \UIA:

\smallskip
\noindent\UIA: {\it For every space $X$ in some class $\mathcal P$ and
  every ideal $\mathcal I\subseteq 2^X$ one of the
  following holds for every $x\in X$:

\begin{enumerate}
    \item there is a countable $\mathcal S\subseteq
      \mathcal I$ such that for every infinite sequence $C$ convergent to $x$
      in $X$ there is an $I\in \mathcal S$ such that $C\cap I$
      is infinite,
    
    \smallskip
    
    \item there is a countable $\mathcal H\subseteq
      \mathcal I^+$ such that for every non-empty open $U\subseteq
      X$, $x\in U$ there is an $H\in \mathcal H$ such that $H\setminus
      U\in \mathcal I$.
\end{enumerate}}
It is not dificult to see that any countable space that is either $\kw$ or first countable
satisfies \UIA. Theorem~\ref{ctble-sequential} shows \IIA\ implies
\UIA\ holds for the class of all countable sequential groups
(note that there are no restrictions on the ideal whatsoever, not even
invariance). The authors do not know at the moment if \UIA\ for all \emph{groomed}
groups is implied by \IIA\ or even whether it is consistent. There are
countable Fr\'echet spaces for which \UIA\ fails in ZFC with the ideal
of the nowhere dense subsets as the witness (see~\cite{dow-pi.frechet}).

To shed some light on the topology of groomed spaces the following more
detailed treatment of the concept of a \vD-point from Lemma~\ref{scattered.vD} may be helpful.

\begin{definition}
  Let $X$ be a topological space. Let $\D$ be a countable family of
  infinite closed discrete subspaces of $X$. We call $\D$ a
  \emph{(strict) \vD-network} at $x\in X$ if for every open $U\ni x$ there
  is a $D\in\D$ such that $D\cap U$ is inifinite ($D\ain U$).
\end{definition}

If $\D$ is a (strict) \vD-network at $x$ we will refer to the space
$\bigcup\D\cup\{x\}$ as a \emph{(strict) \vD-subspace of $X$} and the point $x$
as a \emph{(strict) \vD-point of $\D$ in $X$}.

Now the lemma below offers a topological descripton of groomed spaces.
We omit an elementary proof.

\begin{lemma}\label{groomed}
A countable topological space $X$ is groomed if and only if for every dense
$D\subseteq X$ there exists a point $x\in X$ such that there either exists an infinite sequence
$S\subseteq D$ such that $S\to x$ or a strict \vD-network $\D\subseteq2^D$ at $x$.
\end{lemma}

As indicated by Corollary~\ref{products} the classes of (countable) Fr\'echet
and sequential non Fr\'echet groups are both finitely productive
assuming \IIA\ holds. It appears the following question is open,
including the intriguing possibility of the negative answer in ZFC.
\begin{question}
Does there exist a (separable, or even countable) sequential non Fr\'echet group with a non
sequential square?
\end{question}

The answer to the next question is known to be independent of ZFC for
separable groups, while the non separable case in ZFC remains
open. 
\begin{question}
Does there exist a Fr\'echet group with a non Fr\'echet square?
\end{question}

Example~\ref{cw.not.kw} shows that the dichotomy
of Theorem~\ref{ctble-sequential} does not hold for general sequential
groups even in the separable case. However, the more general question below
appears to be open.
\begin{question}
Does there exist a separable sequential group that is neither $\cw$
nor metrizable?
\end{question}

Finally, while the class of $k_\omega$ spaces is finitely productive,
it is not clear if the same is true about sequential $c_\omega$ spaces. 
\begin{question}
Are sequential \cw-spaces preserved by finite products?
\end{question}

\section*{Acknowledgements}

The authors would like to thank the anonymous referees for their
careful reading of the paper and many helpful suggestions that
significantly improved the exposition. In particular, one of the
referees noted that our proof of the consistency of the
\IIA\ implicitly used a closedness-like assumption for the sets that
generate the ideal. This assumption was made explicit in the
definition of the weakly closed ideal.

The second author (AS) would also like to thank the first author and
Centro de Ciencias Matem\'aticas in Morelia, M\'exico for their
support and hospitality during his visit in October of 2019 which
served as a starting point for the research in this paper.

\end{document}